\documentclass{amsart}
\usepackage[utf8]{inputenc}
\usepackage[T1]{fontenc}
\usepackage{fullpage, verbatim}
\usepackage{graphicx}
\usepackage{amsmath, amsfonts}
\usepackage{amsthm}
\usepackage{amsfonts}
\usepackage{epsfig}
\usepackage{indentfirst}
\usepackage{psfrag,epsf}
\usepackage{amssymb, textcomp}
\usepackage[all]{xy}
\usepackage{caption}
\usepackage{rotating}
\usepackage{tikz}
\usepackage{latexsym}

\theoremstyle{plain}
\newtheorem{theorem}{Theorem}[section]
\newtheorem{corollary}[theorem]{Corollary}

\newtheorem{lemma}{Lemma}[section]

\newtheorem{example}{Example}

\DeclareMathOperator{\mex}{mex}

\theoremstyle{definition}
\newtheorem{definition}{Definition}[section]

\theoremstyle{remark}
\newtheorem{remark}{\textbf{\textit{Remark}}}

\theoremstyle{remark}

\numberwithin{equation}{section}
\bibliography{SourcesComplementarySequence}

\author{Geremias Polanco}

\title{Decomposition of Beatty and Complementary Sequences}

\begin{document}

\begin{abstract}
  In this paper we express the difference of two complementary Beatty sequences, as the sum of two Beatty sequences closely related to them. In the process we introduce a new Algorithm that generalizes the well known Minimum Excluded algorithm and provides a way to generate combinatorially any pair of complementary Beatty sequences. 
  \end{abstract}

\maketitle

\normalsize

\section{Introduction and Statement of parts (1) and (2) of the main theorem}\label{intro}

Wythoff \cite{wythoff1907modification} used 
\begin{equation}
\label{wythoff}
\left[ n \phi^2 \right]- \left[ n \phi \right] =n\text{,}
\end{equation}
where $\phi$ is the golden ratio, in the analysis of a generalized game of Nim (here $\left[x \right]$ denotes the largest integer not exceeding $x$). Later (for instance in \cite{fraenkel_1982} and  \cite{holladay1968some}) it was shown that $$\left[ n \alpha^2 \right]- \left[ n \beta \right] =tn\textbf{,}$$ for $\alpha=(2-t+\sqrt{t^2+4})/2$, and $\beta$ satisfying $\alpha^{-1}+\beta^{-1}=1$. Other generalizations have been done covering the Beatty sequences parametrized by limited families of irrational numbers. We obtain a generalization of identity \eqref{wythoff} for all complementary sequences.
For example, $$\left[ n \frac{3+\sqrt{3}}{2} \right]-\left[ n \sqrt{3} \right]=\left[ n \frac{\sqrt{3}-1}{2} \right]+\left[ n (2-\sqrt3) \right]+1.$$
 In particular, we prove Theorem 1.2 below, which has the following as a corollary.
\begin{corollary}
\label{corollary1}
Let $1<\alpha<\beta$ be irrational numbers with $\frac1\alpha+\frac1\beta=1$. Then
$$\left[ n \beta \right]-\left[ n \alpha \right]=\left[ n \left(\frac{\beta}\alpha-1\right) \right]+\left[ n \left(1-\frac\alpha\beta \right)\right]+1.$$
\end{corollary}

We will give a combinatorial proof of Theorem 1.2.  Before stating the theorem, we fix the following notation. We say that $A$ and $B$ are complementary sets of natural numbers if $A\cap B=\emptyset$ and $A\cup B = {\mathbb N}$. For any multiset $X$ of integers that is bounded below, we number the elements (with multiplicity) of $X$ as $x_1\le x_2\le \dots$, linking in this way indexed lower-case letters (the elements in nondecreasing order) and the upper-case letters (the set). We use set subtraction and addition to denote the termwise operations:
$$X \pm Y = \{x_n \pm y_n : n \in {\mathbb N}\}.$$
A set whose terms are given by $b_n = \left[ n \beta \right]$ for some $\beta>0$ is called a Beatty sequence. We remind the reader of Beatty's Theorem: The sets $A,B$ given by $a_n=\left[ n \alpha \right]$ and $b_n=\left[ n \beta \right]$ are complementary sets of natural numbers if and only if $\alpha$ is irrational and $\frac1\alpha+\frac1\beta=1$. 

\begin{theorem}
[Partial Decomposition of  Differences of Complementary Sequences Parts (1) and (2)]
\label{partial2}
Suppose that $A$ and $B$ are any pair of complementary sequences. Then
\begin{enumerate}
\item The sequence of difference $A-B$, term by term, can always be decomposed in terms of the sum of two other sequences $C$ and $R$ in an easy to see way as follows: 
\begin{equation}
\label{mespropo1}
b_n-a_n=c_n+r_n+1\text{,}
\end{equation}

\noindent where $c_n$ counts the number of integers strictly between $a_n$ and $b_n$ that belong to the sequence $A$, and $r_n$ count the number of integers strictly between $a_n$ and $b_n$ that are in the sequence $B$.  
\item If the two complementary sequences $A$ and $B$ are also Beatty sequences given by $\alpha,\,\beta$ respectively, with $1<\alpha<2$, then $C$ and $R$ in equation \eqref{mespropo1} are Beatty sequences whose $nth$ terms are given respectively by the slopes $\gamma=\frac{2-\alpha}{\alpha-1}=\beta/\alpha -1$ and $\rho=2-\alpha=1-\alpha/\beta$. 

\end{enumerate}
\end{theorem}

The paper is outlined as follows. In the next section, we provide a short proof of corollary \ref{corollary1} that is independent of the combinatorial proof promised above for theorem \ref{partial2}, ending the section by highlighting the importance of the subject. For a deeper understanding of Theorem 1.2 that links to our combinatorial interpretation, in Section 3 we state the MEX algorithm, familiar from the theory of impartial games, and generalize it introducing the novel MES algorithm. In Section 4, we prove the full version of the partial decomposition theorem (which includes Theorem \ref{partial2}) and explore further implications of the MES algorithm.

\section{A Short independent proof of Corollary \ref{corollary1}}\label{alternativeproof}

Before providing an easy proof for Corollary \ref{corollary1}, we want to highlight two important facts. 1) The proof gives no indication as to where the values $\gamma= \frac{2-\alpha}{\alpha-1}$ and $\rho=2-\alpha$ used in this theorem are coming from. It simply shows these two values work, but the proof does not tell us the intrinsic reasons why they work. 2) This easy proof leaves us with no way to connect the sequence $c_n$ and $r_n$ combinatorially with the sequences $A$ and $B$ respectively. In contrast the Partial Decomposition Theorem links $\gamma$ and $\rho$ back to $A$ and $B$, and gives us the natural combinatorial interpretation of $c_n$ and $r_n$. 


\begin{proof}[Proof of Corollary \ref{corollary1}]
We would like to prove that  $k:=[ \beta n ]$ is equal to the sum of the integers 
$t:=[ \alpha n ]$, \; $r:=[ \frac{2-\alpha}{\alpha -1} n +1]$ and $q:=[ (2-\alpha) n ]$.
Note that from $\displaystyle \frac{1}{\alpha}+\frac{1}{\beta}=1$,
we have
\small{\begin{equation} 
\label{mes5}
\frac{2-\alpha}{\alpha-1}=\frac{2-\alpha}{\alpha(1-\frac{1}{\alpha})}=\frac{(2-\alpha)\beta}{\alpha}=(2-\alpha)(\beta-1)=2\beta-2-\alpha\beta +\alpha=2\beta-2-\alpha-\beta +\alpha=\beta-2\text{.}
\end{equation}}
 
Now, we know that
\begin{align} 
&[ \beta n ]=k,  &\mbox{if and only if} \hspace{1cm} &k< \beta n  < k+1\text{,} \label{eq:mes5}\\
&\left[\frac{2-\alpha}{\alpha -1} n +1\right]=r  &\mbox{if and only if} \hspace{1cm}&  r< \frac{2-\alpha}{\alpha -1} n +1 < r+1\text{,} \; \label{eq:mes6}\\
&[ \alpha n ]=t  &\mbox{if and only if} \hspace{1cm} & t<\alpha n < t+1\text{,} \hspace{1cm}\mbox{and}  \label{eq:mes7}\\
&[ (2-\alpha) n ]=q   &\mbox{if and only if} \hspace{1cm}& q< (2-\alpha) n  < q+1 \text{.}\label{eq:mes8}
\end{align}

It follows from \eqref{eq:mes8} that $2n-q-1< \alpha n <2n-q$ and thus, by \eqref{eq:mes7}, $2n-q-1=t\text{,}$ or 
\begin{equation}
 \label{mes9}
2n-1=t+q\text{.}
\end{equation}
From \eqref{mes5} and \eqref{eq:mes6} we find that $r+2n-1<\beta n <2n+r$, hence from 
\eqref{eq:mes5} we deduce that $2n-1+r=k$ and this together with \eqref{mes9} gives $k=t+q+r$.
\end{proof}

To prove the results of this paper, we use combinatorial arguments and standard properties of Beatty and Sturmian sequences. For $\alpha>1$, the Sturmian sequence with slope $1/\alpha$ encodes the Beatty sequence with slope $\alpha$ \cite[Lemma 9.1.3 ]{allouche2003automatic}. There is also a corresponding relationship between non-homogeneous Beatty and Sturmian sequences \cite[Lemma 1]{bosmadekkingsteiner2018}. Sturmian sequences in general, and Beatty sequences in particular, are the focus of a growing amount of research, as they play a role in various fields of mathematics \cite{obryant2002generating}, biology \cite{hata2014neurons}, music \cite{noll2008sturmian}, computer science \cite{bruckstein1991self}, and physics (see also \cite{allouche2003automatic}, \cite{lothaire2002algebraic}, \cite{berstel1996recent},  \cite{blanchet2013counting} and the references therein). The name Sturmian was introduced by Hendlund and Morse in their influential work in the 1940's \cite{morse1940symbolic}, however, the history of these sequences dates back to 1772 when Bernoulli III worked with what is now known as the inhomogeneous Sturmian sequence (see \cite{allouche2003automatic}). In pure mathematics, Sturmian and Beatty sequences have been studied in relation to dynamical systems \cite{morse1940symbolic}, fix point morphisms \cite{stolarsky1976beatty}, logarithm of irrational numbers \cite{polanco2017logarithm}, prime numbers \cite{banks2007prime}, algebraic numbers \cite{fraenkel1994iterated}, arithmetic functions \cite{abercrombie2009arithmetic}, primitive roots, divisor functions, character sums, among others (see \cite{abercrombie1995beatty},   \cite{allouche2019generalized}, \cite{ballot2017functions},  \cite{hildebrand2018almost}, \cite{obryant2002generating}, \cite{porta1990half}, \cite{tang2010prime}, and the references therein). More recently, the second and third moment of Beatty sequences and its squares have also been studied \cite{luca2018variation}. We study differences of Beatty sequences in relation to the well known Minimun Excluded Algorithm that is widely used in combinatorial game theory, coloring algorithm, and else where (see for instance \cite{andrews2019partitions}, \cite{brettell2022excluded}, \cite{conway2000numbers}, \cite{illingworth2022product} and \cite{welsh1967upper}).

\section{Minimun Excluded (MEX) and Minimun Excluded with Skipping (MES) Algorithms}\label{messalgorithm}

We present the well known MEX algorithm, as well as definitions and notations necessary to state the more complete version of the main theorem at the end of this section.

\begin{definition}
\label{definitionmex}
Given a set $S\subset \mathbb{N}$  the $\mex$ of the set $S$ is defined as $\mex(S):=\min  (S^c\cap\mathbf{N})$, i.e., the $\mex(S)$ is the minimum natural number that is not in $S$.  
\end{definition}



It is well known that complementary sequences can be generated using the $\mex$ rule above (see for instance \cite{fraenkel1994iterated}). We call this process the MEX algorithm, and define it more precisely below. In the rest of the paper, if $A$ is a sequence we use $A_n$ to denote the set $A_n=\{ a_k\; :\; k\le n\}$. 

\begin{definition}[Mininmun Excluded Algorithm MEX]
The following set of algorithmic recursions is known as the MEX algorithm, and can be used to generate complementary sequences. The input is a sequence ``H''.\\
$\bullet$ STEP 1: take $a_1=1$,\\
$\bullet$ STEP 2: if $n\ge2$, take $a_n=\mex \{ a_i,b_i\,|\, i<n\}=\mex(A_{n-1}\cup B_{n-1})$, and\\
$\bullet$ STEP 3: if $n\ge1$, take $b_n=a_n+h_n$.\\
$\bullet$ STEP 4: repeat steps $2$ and $3$.\\
\end{definition}

To our knowledge, the first documented use of the MEX algorithm to generate Beatty sequences combinatorially was by British mathematician Willem A. Wythoff in 1907 \cite{wythoff1907modification}. He did so with the purpose of generalizing the combinatorial game Nim. He showed that the winning positions of his game were given by the Beatty pairs $(a_n,b_n)$, defined by the golden ratio. In our context, this corresponds to taking $h_n=n$ in the MEX algorithm above. For the purpose of illustration, let us run the first few rounds of the algorithm with $h_n=n$. We apply STEPS 1, 2 and 3 and get $a_1=1$, and $b_1=a_n+n=1+1=2$, so $A_1=\{1\}$ and $B_1=\{2\}$. Applying STEP 4,  we first obtain $a_2=\mex(A_1\cup B_1)=3$ and $b_2=a_n+n=3+2=5$, so that $A_2=\{1,3\}$ and $B_2=\{2,5\}$. If we follow this process the first five iterations will give Table \ref{messtable4}.\\

\begin{table}[h]
\begin{tabular}{l l l l}
\hline
$\mex(A_{n-1}\cup B_{n-1})$   &   $b_n=a_n+n$   &   $A_n=\{a_k:k\le n \}$ &  $B_n=\{b_k:k\le n \}$\\\hline
$a_1=1$ & $b_1=1+1=2$  & $A_1=\{1\}$   & $B_1=\{2\}$\\\hline
$a_2=3$ & $b_2=3+2=5$  & $A_2=\{1,3\}$   & $B_2=\{2,5\}$\\\hline
$a_3=4$ & $b_3=4+3=7$  & $A_3=\{1,3,4\}$   & $B_3=\{2,5,7\}$\\\hline
$a_4=6$ & $b_4=6+4=10$  & $A_4=\{1,3,4,6\}$   & $B_4=\{2,5,7,10\}$\\\hline
$a_5=8$ & $b_5=8+5=13$  & $A_5=\{1,3,4,6,8\}$   & $B_5=\{2,5,7,10,13\}$\\\hline
\end{tabular}
\caption{first few iterations of the MEX algorithm }
\label{messtable4}
\end{table}

Note that the set $A_n$ agrees with the first $n$ terms of the sequence $A$ given by the golden ratio. Similarly, $B_n$ agrees with the first  $n$ terms of its complementary Beatty sequence. \\

 We now generalize the MEX algorithm by modifying STEP 3, and call this generalization the Minimun Excluded with Skipping (MES) Algorithm. To find $b_n$, we choose the minimum excluded element, but after a number of integers are skipped. We base the MES in the following generalization of definition \ref{definitionmex}

\begin{definition}
Given a set of integers $A$ and a nonnegative integer $k$ we define a function denoted by $\mex_k(A)$ that consists on selecting the $(k+1)$st minimun excluded positive integer from the set $A$. In other words, to find $\mex_k(A)$ we skip the first $k$ excluded integers from $A$ and select the next excluded one.
\end{definition}

\begin{example}
For instance if $A=\{2,4,6,8,\dots\}$, then $\mex_3(A)=7$, since $1,3 \text{ and } 5$ are the first $3$ excluded integers from $A$, and $7$ is the next one excluded. Note that with the above definition, $\mex_0(A)=\mex(A)$, i.e., if we don't have to skip any excluded integers, we are simply applying the $\mex$ rule described above. 
\end{example}


\begin{definition}[Minimun Excluded with Skipping MES Algorithm]
Given a sequence of nonnegative integers $C=\{c_n\}_{n=1}^{\infty}$, the MES algorithm is given by the following steps
\begin{itemize}
\item STEP 1: set $a_1=1$, $A_1=\{a_1\}=\{1\}$ and $B_0=\emptyset$. 
\item STEP 2: for $n\ge2$, $a_n=\mex(A_n\cup B_{n})$. 
\item STEP 3: for $n\ge 1$, $b_n=\mex_{c_n}(A_n\cup B_{n-1})$, i.e., skip the first $c_n$ excluded positive integers and select the next excluded integer from this union.
\item STEP 4: iterate steps 2 and 3.
\end{itemize}
\end{definition}

\begin{example}
\label{goldenMESexample}
We illustrate this algorithm using $C=\{0,1,1,2,3,3,4,4,5,6,6,7,8,8\dots\}$. We start by setting $B_0=\emptyset$. For the first iteration $a_1=1$, and $b_1=\mex_{c_1}(A_1\cup B_0)=\mex_0\{ 1\}=2$. So, $A_1=\{1\}$ and $B_1=\{2\}$. The second iteration gives $a_2=\mex(A_1\cup B_1)=3$, and $b_2=\mex_{c_2}(A_2\cup B_1)=\mex_1\{1,2,3\}=5$, since $5$ is the second minimun excluded integer, i.e., we skip the first excluded integer, $4$, and select the next, $5$. Next iteration gives $a_3=\mex(A_2\cup B_2)=4$, and $b_3=\mex_{c_n}(A_3\cup B_2)=\mex_1\{1,2,3,4,5\}=7$, since we skip the first minimun excluded integer $6$, and select the next excluded integer $7$. If we continue in this process, we get the rows in Table \ref{messtable5}. 
 \end{example}

\begin{table}[h]
$C=\{0,1,1,2,3,3,4,4,5,6,6,7,8,8,9,9,10\dots\}$\\
$a_n=\mex (A_{n-1}\cup B_{n-1})$\\
$b_n=\mex_{c_n}(A_n\cup B_{n-1})$\\
\begin{tabular}{l l l l}
\hline
$a_n$   &   $b_n$   &   $A_n=\{a_k:k\le n \}$ &  $B_n=\{b_k:k\le n \}$\\\hline
$a_1=1$ & $b_1=\mex_0=2$  & $A_1=\{1\}$   & $B_1=\{2\}$\\\hline
$a_2=3$ & $b_2=\mex_1=5$  & $A_2=\{1,3\}$   & $B_2=\{2,5\}$\\\hline
$a_3=4$ & $b_3=\mex_1=7$  & $A_3=\{1,3,4\}$   & $B_3=\{2,5,7\}$\\\hline
$a_4=6$ & $b_4=\mex_2=10$  & $A_4=\{1,3,4,6\}$   & $B_4=\{2,5,7,10\}$\\\hline
$a_5=8$ & $b_5=\mex_3=13$  & $A_5=\{1,3,4,6,8\}$   & $B_5=\{2,5,7,10,13\}$\\\hline
$a_6=9$ & $b_6=\mex_3=15$  & $A_6=\{1,3,4,6,8,9\}$   & $B_6=\{2,5,7,10,13,15\}$\\\hline
$a_7=11$ & $b_7=\mex_4=18$  & $A_7=\{1,3,4,6,8,9,11\}$   & $B_7=\{2,5,7,10,13,15,18\}$\\\hline
$a_8=12$ & $b_8=\mex_4=20$  & $A_8=\{1,3,4,6,8,9,11,12\}$   & $B_5=\{2,5,7,10,13,15,18,20\}$\\\hline

\end{tabular}
\caption{first few iterations of the MEX algorithm }
\label{messtable5}
\end{table}

\begin{remark}
\label{remark3}
Because of how it is used in STEP 3 of the MES definition above, we call $C$ the skipping sequence. After a close look at the resulting Table \ref{tab:t1}, one can see that the sequence $C$ used in the MES above can be defined inductively by the rule: a positive integer that has been assigned by the algorithm to $A$ appears twice in $C$, otherwise it appears once. In other words, the algorithm can be generated without explicitly listing the sequence $C$, only using a) the initial value $c_0=0$, and b) the inductive rule just stated. Tables \ref{tab:t0}, \ref{tab:t2} and \ref{tab:t1} illustrate how $C$ is generated inductively in this way. Later we expand on this showing that the resulting sequence from the MES algorithm gives the complementary pair given by the golden ratio. This is Theorem \ref{partial1} (4a). 
\end{remark}

\begin{center}

\begin{table}
\parbox{.3\linewidth}{
\centering
\caption{}
\label{tab:t0}       
 \begin{tabular}{l l l l}
 \hline\noalign{\smallskip}
$n^{th}$  &   A   &   B   &   C \\
\noalign{\smallskip}\hline\noalign{\smallskip}
 1)  &   1   &   2   &   0 \\
 2)  &       &       &   1 \\
 3)  &       &       &   1\\
 4)  &       &       &   2\\
 \noalign{\smallskip}\hline
 \end{tabular}
 }\hfill
 \parbox{.3\linewidth}{
 \centering
 \caption{}
 \label{tab:t2}
  \begin{tabular}{l l l l}
 \hline\noalign{\smallskip}
$n^{th}$  &   A   &   B   &   C \\
\noalign{\smallskip}\hline\noalign{\smallskip}
 1)  &   1   &   2   &   0 \\
 2)  &   3   &   5   &   1 \\
 3)  &   4   &   7   &   1 \\
 4)  &       &       &   2 \\
 5)  &       &       &   3 \\
 6)  &       &       &   3 \\
 7)  &       &       &   4 \\
 8)  &       &       &   4\\
 9)  &       &       &   5 \\
\noalign{\smallskip}\hline
 \end{tabular}
 }\hfill
  \parbox{.3\linewidth}{
  \centering
  \caption{}
  \label{tab:t3}
  \begin{tabular}{l l l l}
 \hline\noalign{\smallskip}
$n^{th}$  &   A   &   B   &   C \\
\noalign{\smallskip}\hline\noalign{\smallskip}
 1)  &   1   &   2   &   0 \\
 2)  &   3   &   5   &   1 \\
 3)  &   4   &   7   &   1 \\
 4)  &   6   &   10  &   2 \\
 5)  &   8   &   13  &   3 \\
 6)  &   9   &   15  &   3 \\
 7)  &   11  &   18  &   4 \\
 8)  &   12  &   20  &   4 \\
 9)  &   14  &   23  &   5 \\
 10)  &   16  &   26  &   6 \\
 11)  &   17  &   28  &   6 \\
 12)  &   19  &   31  &   7 \\
\noalign{\smallskip}\hline
 \end{tabular}
\label{tab:t1}
}\hfill
\end{table}
\end{center}
\begin{definition}
\label{sortjoin}
    Define the sortjoin of two integer sequences $A$ and $B$, denoted by $A\star B$, as the union with repetition of the ordered elements of these two sequences.
\end{definition} 

We use $A^2$ for the sortjoin of $A$ with itself i.e., $A^2=A\star A$, and we represent the sortjoin of $k$ copies of $A$ with itself by $A^k$. Finally, we use $\mathbb{Z}_{\{0\}}$ to denote the nonnegative integers. For instance if $A$ represents the even natural numbers, then $A\star \mathbb{Z}_{\{0\}}^2=\{0,0,1,1,2,2,2,3,3,4,4,4,5,5,\cdots\}$.

We need two more definitions before proving the main theorem. Part (4) of that theorem (below) mentions the case when a sequence $C$ is given by the sortjoin $C=D\star \mathbb{Z}_0^{k-1}$ in the MES algorithm. In this case, the sequence $C$ is completely determined by the sequence $D$. Since $C$ defines the MES algorithm we effectively have that $D$ defines the MES algorithm. We record this in the following definition.
 
 \begin{definition}
 If there is an increasing sequence of natural numbers $D$ such that the sequence $C$ of the MES algorithm is given by $C=D\star \mathbb{Z}_0^{k-1}$, we call $D$ the defining sequence of the MES algorithm.
 \end{definition}

We also need the following:
\begin{definition}
\label{mescountingsequence}
For $n\ge 1$, consider the (possibly empty) half-open integer interval $(a_n,b_{n-1}]$, and let $r_n$ be the cardinality of the intersection of this interval with $B_{n-1}$, i.e., $r_n=\#\{ k\;|\; k\in (a_n,b_{n-1}] \text{ and } k \in B_{n-1}\}$. We define the auxiliary sequence $R$ by $\displaystyle R:=\{r_n\}_{n=1}^{\infty}$. Also we define $R_n:=\{r_k\;:\; k\le n\}$.
\end{definition}

\section{Partial Decomposition of Beatty sequences and the MES algorithm: Proof of the main theorem}\label{partialandbeatty}

We now state the main theorem, and include parts (1) and (2) stated in Section 1 for completion.

\begin{theorem}[Partial Decomposition of  Differences of Complementary Sequences, Full Version]
\label{partial1}
Suppose that $A$ and $B$ are any pair of complementary sequences. Then
\begin{enumerate}
\item The sequence of difference $A-B$, term by term, can always be decomposed in terms of the sum of two other sequences $C$ and $R$ as follows: 
\begin{equation}
\label{mesprop4}
b_n-a_n=c_n+r_n+1\text{,}
\end{equation}

\noindent where $c_n$ counts the number of integers strictly between $a_n$ and $b_n$ that belong to the sequence $A$, and $r_n$ count the number of integers strictly between $a_n$ and $b_n$ that are in the sequence $B$.  
\item If the two complementary sequences $A$ and $B$ are also Beatty sequences given by $\alpha\text{ and }\beta$ respectively, with $1<\alpha<2$, then $C$ and $R$ in equation \eqref{mesprop4} are Beatty sequences whose $nth$ terms are given respectively by the slopes $\gamma=\frac{2-\alpha}{\alpha-1}$ and $\rho=2-\alpha$. 
\item There is a combinatorial interpretation of formula \eqref{mesprop4}, that is, an algorithm which we name the Minimun Excluded with Skipping Algorithm (MES), that takes as input a sequence $C$ and it outputs three sequences $A, B$ and $R$ such that if $a_n, b_n, c_n$ and $r_n$ are the $nth$ term of $A,B,C$ and $R$ respectively, then equation \eqref{mesprop4} holds, i.e., 
$$ b_n-a_n=c_n+r_n+1\text{.}$$
Furthermore, the MES characterizes complementary sequences. Specifically, $A$ and $B$ are complementary sequences, if and only if,  there is a nonnegative sequence $C$ that generate $A$ and $B$ through the MES algorithm.
\item The MES generalizes the MEX algorithm as follows: suppose there is an increasing sequence of positive integers $D$ such that $C=D\star \mathbb{Z}_0^{k-1}$ for some $k\in \mathbb{N}$. Then we have the following
\begin{enumerate}
\item If  $k=2$ and $D$ is the Beatty sequence defined by the golden ratio, then the MES and the MEX algorithm coincide, i.e., they both produce $A$. In other words, the MES produces the sequence $A$ if and only if the skipping sequence $C$ can be defined as follows: any nonnegative integer $c$ appears  once or twice in $C$, and $c$ appears twice in $C$ if and only if $c\in A$.
\item  If $D$ is the sequence given by $\alpha$ where 
$$\alpha=\frac{k-1+\sqrt{k^2+1}}{k}\text{,}$$
then the MES algorithm produces the Beatty sequence $D$ and its complement. In other words, the MES produces the sequence $D$ if and only if the skipping sequence $C$ can be defined as follows: any nonnegative integer $c$ appears  $k$ or $k-1$ times in $C$, and $c$ appears $k$ times in $C$ if and only if $c\in D$. This coincides with the MEX and the golden ratio when $k=2$.
\end{enumerate}
\end{enumerate}
\end{theorem}

In this section, we are going to prove the Partial Decomposition Theorem. Parts (1) and (3) are the simplest to prove, and we do so immediately. We complete the proof of Part (2) in Theorems \ref{mesprop6} and \ref{slopeofC}. Finally, we prove Part (4) of the main theorem in Corollary \ref{corofinal}.

\begin{proof}[Proof of Parts (1) and (3) of Theorem \ref{partial1}]
Part 1 of Theorem \ref{partial1} is evident, by virtue of the fact that $A$ and $B$ are complementary, and thus, the integers in the interval $(a_n,b_n)$ will belong either to $A$ or $B$, hence the stated formula in part 1 must follow. Part 3 of the theorem follows by construction of the MES algorithm. On the one hand, given a nonnegative sequence $C$ the MES can be implemented, skipping $c_n$ excluded integers in the $nth$ step. By construction it will produce complementary sequences. Conversely, if $C$ is negative, the skipping cannot be performed in the natural sense that is done in the algorithm. On the other hand, if a pair of complementary sequences is given, one can define $r_n$ as in Definition \ref{mescountingsequence}.  Then $c_n$ can be defined as the following complement: $[(a_n,b_n)\cap B_n]^c$. With this set up, we can implement the algorithm with $C=\{ c_n\}_{n=1}^{\infty}$. From which it will clearly follow that if $R=\{ r_n\}_{n=1}^{\infty}$, then $b_n-a_n=c_n+r_n+1$.  
\end{proof}
We spend the rest of this section proving parts 2) and 4) of Theorem (\ref{partial1}). We do so through a sequence of lemmas regarding Sturmian and Beatty sequences. Given a Beatty sequence with slope $0<\alpha<1$, each positive integer will occur in the sequence $k$ 
times or $k+1$ times, for some number $k\ge 1$. This is simply a consequence of the size of $\alpha$. We have the following:

\begin{lemma}
\label{meslem9}
If $\alpha$ belongs to the interval $\frac{1}{k+1}<\alpha < \frac{1}{k}$ for $k$ a positive integer, then all nonnegative integers appear in the Beatty sequence $C$ with slope $\alpha$, each of them repeating $k$ or $k+1$ times. Both occur infinitely often. In the language of sortjoin, there exists, in this case, a sequence $B$ of nonnegative integers such that $C=B\star\mathbb{Z}_0^{k-1}$.
\end{lemma}

\begin{proof}
It's not hard to see that the inequality $\frac{1}{k+1}<\alpha < \frac{1}{k}$ implies each integer repeats $k$ or $k+1$ times in $\{[ \alpha n ] \}_{n=1}^{\infty}$. Indeed, $k\alpha <1 $ implies each integer 
$k$ occurs at least $k$ times, and an integer can not occur $k+2$ times nor more, for if 
$[ n \alpha ]= [ (n+1) \alpha ]=\cdots=[ (n+k+1) \alpha ]=k$, then 
$0=[ (n+k+1) \alpha ]-[ n \alpha ] \ge [ n \alpha ]+ [ (k+1) \alpha ]-[ n \alpha ]=[ (k+1) \alpha ]>1$.
Now some integers $n$, will indeed need to repeat $k+1$ times, for if all large integers repeat $k$ times, 
then the sequence would be ultimately periodic and then the following equality would need to hold:
$$\displaystyle \lim_{n \to \infty} \frac{[ n \alpha ]}{n}=\frac{1}{k}\text{,}$$
 i.e., $\alpha=1/k$, a contradiction to the irrationality of $\alpha$. For the same reasons it cannot repeat $k+1$ times.
 Hence both $k$ and $k+1$ occur infinitely often.
\end{proof}

\begin{definition}
 \label{thbrassdef1} 
     For a real number $\alpha$, such that $0<\alpha<1$, we define the
characteristic function of $\alpha$ as
           \begin{equation}
               \label{thbrass1} 
                  f_{\alpha}(n)= [ \alpha (n+1) ]-
[ \alpha n ] \text{.}
           \end{equation}
\end{definition}

Clearly $f_{\alpha}(n)=1$ or $f_{\alpha}(n)=0$.  Since the sum
telescopes, we have
\begin{equation} 
\label{thbrass2}
\sum_{n=1}^{m} f_{\alpha}(n) =[ \alpha (m+1) ]\text{.}
\end{equation}
\begin{remark}
Notice that \eqref{thbrass2} gives the number of integers $n\leq m$  for
which $f_{\alpha}(n)=1\text{.}$
\end{remark}
\begin{definition} 
 \label{thbrassdef2}
For $1<\beta \in \mathbf{R \setminus Q}$, define
    \begin{equation}
      \label{thbrass3} 
           g_{\beta}'(n)=
             \begin{cases}
              1 &\textit{if }     n=[ k\beta ],
\textit{for k }\hspace{1mm} \in \hspace{1mm}\mathbf{Z},\\
              0&      \hspace{3mm} \text{otherwise.}
             \end{cases}
    \end{equation}
\end{definition}

It is a well known fact (see \cite{polanco2017logarithm} and \cite[Lemma\ 9.1.3]{allouche2003automatic}) that for all integers $n$
\begin{equation} 
 \label{thbrass4}
 g_{\beta}' (n)= f_{1/\beta} (n)\text{.}
\end{equation}

\begin{lemma}
\label{mesprop2}
For each irrational number $\alpha$ such that $\frac{1}{k+1}<\alpha < \frac{1}{k}$, there exists a number 
$\beta:=\frac{\alpha}{1-\alpha}$, such that $[ n \alpha ]=t$ for $k+1$ different numbers 
$n=m,m+1,...,m+k$, if and only if $[ n \beta ]=t$ for $k$ numbers $n=q,q+1,...,q+k-1$. In other words, if $C=[ n\alpha ]_{n=1}^{\infty}$ and $\displaystyle B=\left[ \frac{\alpha}{1-\alpha}\cdot n \right]_{n=1}^{\infty}$ then $C=D\star \mathbb{Z}_0^{k}$ and $B=D\star \mathbb{Z}_0^{k-1}$ for some increasing sequence $D$ of nonnegative integers.
\end{lemma}

\begin{proof}
Notice that $\frac{1}{k+1}<\alpha < \frac{1}{k}$ if and only if $k < \frac{1}{\alpha}<k+1$, if and only if
 $k-1 < \frac{1}{\alpha}-1<k$ and $\frac{1}{\alpha}-1=\frac{1-\alpha}{\alpha}=\frac{1}{\beta}$. Thus 
we claim that it is enough to prove that the sequence $[ n \alpha ]=t$ for $k+1$ values of $n$, if and only if 
 $[ (t+1) \frac{1}{\alpha} ]-[ t \frac{1}{\alpha} ] =k+1$. Indeed, if this is true, since
 $[ (t+1) \frac{1}{\beta} ]-[ t \frac{1}{\beta} ]= [ (t+1) (\frac{1}{\alpha}-1) ]-[ t (\frac{1}{\alpha}-1) ]=[ (t+1) \frac{1}{\alpha} ]-[ t \frac{1}{\alpha} ]-1 = k$, 
it would follow from this claim that $[ n \beta ]$ repeats $k$ times. 
So we just need to prove the double implication in the claim. For that purpose, from \eqref{thbrass1},  
we write $f(n):=f_{\alpha}(n)=[ (n+1)\alpha ]-[ n \alpha ]$. And from \eqref{thbrass4}, 
it follows that $f(n)=1$ if and only if there exists a $t$ such that $n=[ t \frac{1}{\alpha} ]$. Thus,
 from this last sentence we find that $[ \alpha (n+m) ]=r$ for the first $k+1$ values of $m$ exactly, if and only if 
the following three things happen: First, $f(n-1)=0$; second, $f(n)=f(n+k+1)=1$ and third, $f(n+m)=0$ for all numbers in between $n$ and $n+k+1$, i.e., for all $m$ such that $1\leq m< k+1$. 
For such an $n$, it follows from \eqref{thbrass4}  that $f(n+m)=0$ for $1\leq m< k+1$ if and only if there is no $j$ such 
that $n+m=[ j \frac{1}{\alpha} ]$, for $1 \leq m\leq k$. Thus we see that 
$[ (t+1) \frac{1}{\alpha} ]-[ t \frac{1}{\alpha} ] > k$. Since the difference between two consecutive 
numbers in the Beatty sequence $[ t \frac{1}{\alpha} ]$ is either $k$ or $k+1$, we see that
 $[ (t+1) \frac{1}{\alpha} ]-[ t \frac{1}{\alpha} ]=k+1$. So, we have proved that each time the
 sequence $[ n \alpha ]$ repeats $k+1$ times, then the difference 
$[ (t+1) \frac{1}{\alpha} ]-[ t \frac{1}{\alpha} ] =k+1$, and thus, the Theorem holds.
\end{proof}

Combining Lemma \ref{meslem9} and Lemma \ref{mesprop2}, we see that if a Beatty sequence is generated by an $\alpha$ such that each nonnegative integer repeats
$k$ or $k+1$ times, then in the Beatty sequence generated by $\frac{\alpha}{1-\alpha}=-1+\frac{1}{1-\alpha}$ integers repeat $k-1$ or $k$ times.
 If we iterate this process we have the following Theorem. 
\begin{lemma}
\label{mesprop3}
 For each number $\alpha$ with $\frac{1}{k+1}<\alpha < \frac{1}{k}$ for some integer $k$, there exists a number $\beta$ with 
$1<\beta<2$ given by 
$\beta -1+\cfrac{1}{2-\cfrac{1}{2-\cfrac{1}{2-\cfrac{...}{...-\cfrac{1}{2-\cfrac{1}{1-\alpha}}}}}}$, and $k-2$ two's 
in this expansion, such that $[ n \alpha ]=k$ for $k+1$ different numbers if and only if $k$ is in the Beaty sequence
 given by $[ n \beta ]$.
\end{lemma}

\begin{proof}
In Lemma \ref{mesprop2}, an irrational $\alpha$ giving a maximum of $k+1$ repetitions, implies that the number 
\begin{equation}
\label{mes4b}
 \alpha_1:-1+\frac{1}{1-\alpha}
\end{equation}
 produces a maximum of $k$ repetitions. We iterate this process a second time and find that
\begin{equation}
\label{mes4c} 
\alpha_2:=-1+\frac{1}{1-\alpha_1}=-1+\cfrac{1}{1-\left(-1+\cfrac{1}{1-\alpha_1}\right)}=-1+\cfrac{1}{2-\cfrac{1}{1-\alpha}}
\end{equation}
produces a maximum of $k-1$ repetitions. If we iterate this process a total of $k$ times, we will get a 
sequence with no repetition, i.e., a Beatty sequence $[ n \beta ]$, for some $\beta>1$. 
Note that in each iteration after the second one, a new digit $2$ will continue to appear
as in the right hand side of \eqref{mes4c}. Thus the slope $\beta$ of the Beatty sequence is necessarily of the form
 $$\displaystyle -1+ \cfrac{1}{2-\cfrac{1}{2-\cfrac{1}{2-\cfrac{...}{-\cfrac{1}{2-\cfrac{1}{1-\alpha}\;\text{.}}}}}}$$
\noindent The number of times the digit $2$ appears in this expression is $k-2$, for the following reasons.
We start the process with the number $\alpha$ giving $k$ and $k+1$ repetition. Next step gives $\alpha_1$ which does not have
the digit $2$ in it by \eqref{mes4b}. The third number in this process, $\alpha_2$, is the first one having the digit $2$, and thus
this digit will appear $k-2$ times in $\beta$.
\end{proof}

\begin{corollary}
 Consider a Beatty sequence with slope $\alpha<1$, and let $k$ be the integer such that $\frac{1}{k+1}<\alpha < \frac{1}{k}$. Then the sequence of numbers that repeat 
$k$ times and the sequence of numbers that repeat $k+1$ times form a partition of the integers given by two complemetary Beatty sequences.
\end{corollary}
\begin{proof}
 Since $0<\alpha<1$, each integer $n$ is in the Beatty sequence with slope $\alpha$, and it is repeated $k$ or $k+1$ times. 
By Lemma \ref{mesprop3} those numbers that repeat $k+1$ times form a Beatty sequence with slope $\beta$, 
hence those that only repeat $k$ times are the complementary Beatty sequence with slope $\displaystyle \frac{1}{1-\frac{1}{\beta}}\text{.}$
\end{proof}

We reinterpret the previous result in terms of sortjoin of sequences, to get
\begin{corollary}
\label{corsortjoin}
 Let $\alpha\in [0,1]$ be an irrational number , and let $k$ be the integer such that $\frac{1}{k+1}<\alpha < \frac{1}{k}$. Then the sequence $C=\left\{[ \alpha n]\right\}_{n=1}^{\infty}$ is given by the $k$ fold sortjoin of $\mathbb{Z}_0$ with a beatty sequence $D$. Specifically, $C=D\star\mathbb{Z}_0^{k}$.
\end{corollary}

\begin{lemma}
\label{mesprop5}
Let  $A$ and $B$ be the complementary Beatty sequences generated by $\alpha$ and $\beta$ respectively,
 with $(1<\alpha<2)$, and let $a_n$ and $b_n$ be the $n^{th}$ elements of $A$ and $B$, respectively. 
Define $r_n$ as in Definition \eqref{mescountingsequence}. Then the difference, $r_{n+1}-r_n$ equals $0$ or $1$, 
and $r_{n+1}-r_n=1$ if and only if $a_{n+1}-a_n=1$, hence the sequence $r_n-r_{n-1}$ is Sturmian.
\end{lemma}

\begin{proof}
To prove the first part of the theorem, notice that by Definition \ref{mescountingsequence}, we are given that, for an integer $n$,  $R_n$ has $r_n$ elements, all of which belong
to $B$. Let us list all the elements of $R_n$, in descending order: 

\begin{equation}
  \label{mes10}
R_n=\{b_{n-1}, b_{n-2},...,b_{n-r_n} \}\text{.}
\end{equation}
\noindent Similarly,
\begin{equation}
\label{mes12}
 R_{n+1}=\{b_{n}, b_{n-1},...,b_{n-r_{n+1}} \}\text{.}
\end{equation}

We can see from these two sets that $R_{n+1}$ always contains an element that is not in $R_n$, namely $b_n$.
 Also note that the smaller elements of $R_n$ and $R_{n+1}$, namely $b_{n-r_n}$ and $b_{n-r_{n+1}}$ respectively, 
could be the same element. Indeed, one can see that if $a_n$ and $a_{n+1}$ are consecutive elements, then 
from Definition \ref{mescountingsequence} we see that $b_{n-r_n}=b_{n-r_{n+1}}$. On the other hand, if $a_n$ and $a_{n+1}$ are not 
consecutive, then since $A$ and $B$ are complementary Beatty sequences, there have to be some element of $B$ between $a_n$ and 
$a_{n+1}$. In fact there can only be one element, since $\frac{1}{\alpha}+\frac{1}{\beta}=1$ and $1<\alpha<\beta$ implies $\beta>2$ and thus $b_{n+1}-b_n\ge2$. Hence, we see that $r_{n+1}-r_n=1$ if and only if $a_{n+1}-a_n=1$, and if the difference $a_{n+1}-a_n\neq 1$, then $r_{n+1}-r_n=0$. Finally, the above argument together with the fact that $a_{n+1}-a_n=1$ or $2$, implies that $r_{n+1}-r_n=a_{n+1}-a_n-1$. Since the sequence given by the right hand side of this equality is Sturmian, it follows that the sequence given by the left hand side is also Sturmian. \\
\end{proof}

In a similar manner we get the following lemma.

\begin{lemma}
\label{CisBeatty}
Let $A$ and $B$ be complementary Beatty sequences and suppose that the slope $\beta>2$ and let $q=[ \beta ]$. Suppose $a_n$ and $b_n$ are the $nth$ term of $A$ and $B$ respectively, and let $c_n=\{ a\in A\,:\, a_n<a<b_n \}$. Then, the difference $c_{n+1}-c_n$ equals $q-2$ or $q-1$, and $c_{n+1}-c_n=q-2$ if and only if $b_{n+1}-b_n=q$.
\end{lemma}

The following is a corollary to Lemma \ref{mesprop5}.
\begin{corollary}
 \label{mescor2}
If $A$ and $B$ are complementary Beatty sequences, then the sequence $R=\{r_n\}_{n=1}^\infty$ is Beatty with slope $r$ such that $0<r<1$.
\end{corollary}

\begin{proof}
 The difference between two elements of $R$ is either $0$ or $1$. Thus if $R$ is Beatty, the slope should be a 
number $r$ with $0<r<1$. Since $A$ is a Beatty sequence, Lemma \ref{mesprop5} together 
with the fact that $r_{n+1}-r_n$ can take only two possible values, implies that $R$ is also Beatty. Hence the corollary holds. 
\end{proof}

We now get the following proposition that is directly related to part 2 of the main theorem. Indeed, Theorem \ref{mesprop6} and \ref{slopeofC} complete the proof of the second part of the main theorem.

\begin{theorem}
\label{mesprop6}
If $A$ and $B$ are are complementary Beatty sequences, then the sequence $R=\{r_n\}_{n=1}^\infty$ is given by the slope $r:=2-\alpha$, where $\alpha$ is the slope of the sequence $A$. 
\end{theorem}

\begin{proof}
Since $0<2-\alpha <1$, in light of Theorem \ref{mesprop5} and its corollary, it is enough to prove 
that $[ (2-\alpha)(n+1) ]-[ (2-\alpha)n ]=1$ if and only if $a_{n-1}-a_n=1$. 
For this purpose, note that 
$[ (2-\alpha)(n+1) ]-[ (2-\alpha)n ]=2(n+1)-2n+[ -\alpha(n+1) ]-[ -\alpha n ]=2+[ -\alpha(n+1) ]-[ -\alpha n ]$. 
It is clear that, for $x$ non integral, $[ -x]= -1-[ x]$, thus we see that 
$2+[ -\alpha(n+1) ]-[ -\alpha n ]= 2-(a_{n+1}-a_n)$. This is equal to $1$ if and only if $a_{n+1}-a_n=1$.

\end{proof}

In a similar manner we get the following theorem, using the definition of $C$ given in Lemma \ref{mesprop5}
\begin{theorem}
\label{slopeofC}
If $A$ and $B$ are complementary Beatty sequences, then, the sequence $C=\{c_n\}_{n=1}^\infty$ is given by the slope $c:=\frac{2-\alpha}{\alpha-1}$, where $\alpha$ is the slope of the sequence $A$. 
\end{theorem}

We get the following theorem, that we find it to be a striking formula involving homogeneous Beatty sequences. It also has a very natural and straightforward interpretation that we highlight in the remark that follows the theorem. This interpretation is highlighted in part 2) of the main theorem.

\begin{theorem}
\label{mesprop6a}
Let $\alpha$ and $\beta$ be two irrational numbers generating two complementary Beatty sequences, 
$A$ and $B$, respectively, and let $1<\alpha<2$. 
Set $a_n=[ \alpha n ]$ and $b_n=[ \beta n]$. Then for any integer $n$ we have
\begin{equation}
\label{mes12a} 
  [ \beta^{-1}(b_{n-1}+1)]-[ \beta^{-1}(a_{n})]=[(2-\alpha)n]\text{.}
\end{equation}
Similarly, we have 
\begin{equation}
\label{mes12b} 
[ \alpha^{-1} b_n ] -[ \alpha^{-1}(a_n+1)]=\left[ \frac{2-\alpha}{\alpha-1}n\right]\text{.}
\end{equation}
\end{theorem}

\begin{proof}
We will prove identity \eqref{mes12a}, since \eqref{mes12b} follows by a similar argument using Theorem \ref{slopeofC}. To prove equation \eqref{mes12a} in light of Theorem \ref{mesprop6}, we just need to prove that 
$r_n= [ \frac{1}{\beta}(b_{n-1}+1)]-[\frac{1}{\beta}(a_{n})]$. 
For that purpose, recall that $r_n$ can be written as: 
\begin{equation}
 \label{mes13}
r_n= \# \{ b \in B: a_n < b \leq b_{n-1}\}= \sum_{\substack{b\in B \\ a_n < b \leq b_{n-1}}} 1 \text{.}
\end{equation}
Now define $f(k):=f_{\frac{1}{\beta}}(k)$ as in \eqref{thbrass1}. 
By \eqref{thbrass4}, $f(a_n)=0$, and thus by the same equation we find that \eqref{mes13} becomes:
\begin{equation}
\label{mes14}
 \sum_{a_n \leq k \leq b_{n-1}}f(k)=\sum_{1 \leq k \leq b_{n-1}}f(k)-\sum_{1 \leq k < a_n}f(k)=[ \beta^{-1}(b_{n-1}+1)]-[ \beta^{-1}(a_{n})]\text{.}
\end{equation}
This last equality follows from \eqref{thbrass2}.
\end{proof}

\begin{remark}
 \label{mesprop6b}
We will interpret equation \eqref{mes12a}, and a similar interpretation follows for equation \eqref{mes12b}. Note that in Corollary \ref{mesprop6a} the Sturmian sequence with slope $1/\beta$ is the
indicator function of the Beatty sequence with slope $\beta$. Hence, in light of Lemmas \ref{thbrass2} and \ref{thbrass4}, the left hand side of equation \ref{mes12a}, 
gives the number of integers strictly between $a_n$ and $b_n$ 
which are in $B$. As a passing comment, we mention that the right hand side of \eqref{mes12a} gives us another way to write this quantity. Since these $r_n$ integers are in $B$, if we want to generate a complementary pair by the MES we simply need to find a sequence $c_n'$ that consistently gives the number of integers between $a_n$ and $b_n$ that are in $A$. We would then take these numbers $c_n'$ as the skipping sequence in the algorithm. This is what we show in the following Theorem. As mentioned before, the last part of Lemma \ref{mesprop8} finishes the proof of the Partial Decomposition Theorem.\end{remark}

\begin{lemma}
 \label{mesprop8}
Let $\alpha$ be an irrational number such that $1<\alpha<2$, and let $k\ge1$ be the unique positive integer
 for which $\displaystyle \frac{2k-1}{k}< \alpha < \frac{2k+1}{k+1}$. The skipping sequence $C$ that generates
 the MES algorithm with output the Beatty sequence given by $\alpha$, can be generated by a Beatty sequence $D$ 
with slope $\delta$ given as follows:

\begin{enumerate}
 \item if $k>1$, (i.e., if $\displaystyle \frac{3}{2}< \alpha <2$), then the slope $\delta$ is given by:
$$\displaystyle \delta:=-1+\cfrac{1}{2-\cfrac{1}{2-\cfrac{1}{2-\cfrac{...}{...-\cfrac{1}{2-\cfrac{1}{\alpha-1}}}}}}$$ with 
$k-1$ two's in this expansion, and the frequency for the skipping sequence is given as follows: Each number $n$ will 
appear $k$ times in $C$ if $n \in D$, and $k-1$ times otherwise.
\item If $k=1$ (the case $\displaystyle 1<\alpha<\frac{3}{2}$), then $\delta=\frac{2-\alpha}{\alpha-1}$, 
and the frequency is given as follows: each number $n$ will appear once in the skipping sequence if $n\in D$ and will not 
appear in the skipping sequence if $n \notin D$.
\item Furthermore, $C$ is a Beatty sequence and is given by the slope $c=\frac{2-\alpha}{\alpha-1}$.
\end{enumerate}
\end{lemma}

\begin{remark}
\label{mesrem6}
 Note here that if $\alpha=\phi$, where $\phi$ is the golden ratio, then 
$c=\frac{2-\phi}{\phi-1}=(2-\phi)\phi=2\phi-\phi^2=\phi-1=\frac{1}{\phi}$.  Also, observe here that the number $\delta$ can be written as
$$\delta =\frac{2-\alpha}{k\alpha -2k+1}\text{.}$$
This reduces to
$$\delta =\frac{2-\alpha}{\alpha -1}$$
 when $k=1$.
\end{remark}

\begin{proof}
We already proved in Theorem \ref{slopeofC} that $C$ is generated by $c:=\frac{2-\alpha}{\alpha -1}$. 
Since $1<\alpha <2$, there exists a unique positive integer $k$ such that 

\begin{equation}
 \label{mes17}
\displaystyle \frac{2k-1}{k}< \alpha < \frac{2k+1}{k+1}\text{.}
\end{equation}

Case I: $k>1$. We will apply Theorem \ref{mesprop3} with the $\alpha$ of that Theorem 
equal to  $c=\frac{2-\alpha}{\alpha-1}=-1+\frac{1}{\alpha-1}$ (i.e., the slope of the sequence $C$). 
We first notice that the last portion of the continued fraction of $\beta$ in Theorem \ref{mesprop3} is given by 
\begin{equation}
 \label{mes18}
2-\cfrac{1}{1-c}=2-\cfrac{1}{2-\cfrac{1}{\alpha-1}}\text{.}
\end{equation}
\noindent Thus, we gain an extra $2$ here, compared to Lemma \ref{mesprop3}, and hence, when applying that lemma we  
will need to adjust for that extra $2$. With that in mind, we see that \eqref{mes17} implies 
that $\frac{1}{k}<\frac{2-\alpha}{\alpha -1}< \frac{1}{k-1}$. Since, there are $k-1$ two's in the expansion of $\delta$, it is 
equivalent to having $k-2$ two's in Lemma \ref{mesprop3}, because we need to adjust for the extra $2$, shown in 
\eqref{mes18}.  Thus each positive integer will appear in $C$ $k$ or $k-1$ times. Also, Lemma \ref{mesprop3} tells us that
those numbers that will appear $k$ times are, precisely, the integers belonging to the Beatty sequence $D$ with the slope 
that we here called $\delta$.

Case II: $k=1$, i.e., $\displaystyle 1<\alpha<\frac{3}{2}$. We see that $c=\frac{2-\alpha}{\alpha -1}>1$. In this case,
 each positive integer either will appear once in $C$ or will not appear at all. Thus $D$ and $C$ turn out to be 
the same sequence, and thus $C$ can be described as follows. Each integer $n$ will appear once in $C$ if  $n\in D$, and will not appear 
in $C$ otherwise. 
\end{proof}

If we compare parts (1) and (3) in Lemma \ref{mesprop8} with Corollary \ref{corsortjoin} we get:
\begin{corollary}
To generate a Beatty sequence with the MES algorithm as in Lemma \ref{mesprop8} the skipping sequence is given by the sortjoin of $D$ and $k$ copies of $\mathbb{Z}_0$. Specifically, $C=D\star\mathbb{Z}_0^k$, and the slope of $C$ is given by $\frac{2-\alpha}{\alpha-1}$
\end{corollary}

\begin{theorem}
 \label{mesprop11}
The MES algorithm generates a Beatty sequence if and only if the defining sequence is also a Beatty sequence  $D$ given 
by the slope $\delta$ from Lemma \ref{mesprop8}, and the frequency is given by: if a number belongs to $D$, it will appear $k$ times in the skipping sequence, otherwise, it will appear $k-1$ times. \end{theorem}

\begin{proof} 
The number $\delta$ given in Lemma \ref{mesprop8} can be also written as 
\begin{equation}
\label{mes21}
\delta =\frac{2-\alpha}{k\alpha -2k+1}\text{,}
\end{equation}
or equivalently
\begin{equation}
\label{mes21a}
\alpha = \frac{2-\delta + 2 k \delta}{1+k\delta}\text{.}
\end{equation}
Equation \eqref{mes21} tell us that to produce the Beatty sequence with slope $\alpha$, we just need to start 
with the skipping sequence generated by the number $\delta$ given in the right hand side of the equality, with frequency $k$. 
We know that if $A$ is a Beatty sequence, so is $B$. So we may assume both $A$ and $B$ are Beatty sequences. In this case, 
Lemma \ref{mesprop8} tells us that the defining sequence $D$ is also a Beatty sequence given by \eqref{mes21}. 
For the converse, \eqref{mes21a} tells us that given a defining sequence with slope $\delta$, we will generate a different 
Beatty sequence with slope $\alpha$ for each positive integer $k$ we take, as given in \eqref{mes21a}.
\end{proof}

We now state and prove the following theorem, which is part (4), of the Partial Decomposition Theorem.
\begin{corollary}
\label{corofinal}
The MES generalizes the MEX algorithm as follows: Suppose there is an increasing sequence of positive integers $D$ such that $C=D\star \mathbb{Z}_0^{k-1}$ for some $k\in \mathbb{N}$. Then we have the following
\begin{enumerate}
\item If  $k=2$ and $D$ is the Beatty sequence defined by the golden ratio, then the MES and the MEX algorithm coincide, i.e., they both produce $A$. In other words, the MES produces the sequence $A$ if and only if the skipping sequence $C$ can be defined as follows: any nonnegative integer $c$ appears  once or twice in $C$, and $c$ appears twice in $C$ if and only if $c\in A$.
\item  Let $D$ be the sequence given by
$$\alpha=\frac{k-1+\sqrt{k^2+1}}{k}\text{.}$$
Then the MES algorithm produces the Beatty sequence $D$ and its complement. In other words, the MES produces the sequence $D$ if and only if the skipping sequence $C$ can be defined as follows: any nonnegative integer $c$ appears  $k$ or $k-1$ times in $C$, and $c$ appears $k$ times in $C$ if and only if $c\in D$. This coincides with the MEX and the golden ratio when $k=2$.
\end{enumerate}
\end{corollary}
\begin{proof}
    Substitute $\alpha=\delta$ in Equation \ref{mes21a}, and solve for $\delta$ in the resulting rational expression to obtain a quadratic irrational with parameter $k$. The case $k=2$ gives the first part of the theorem.  
\end{proof}
\begin{remark}
As seen in \eqref{mes21} and \eqref{mes21a}, a given $\delta$ can generate infinitely many $\alpha$'s. However each 
given $\alpha$ is generated by a unique choice of $\delta$ and $k$.\\
\end{remark}
\noindent \textbf{Example.} As an illustration of Theorem \ref{mesprop11} consider the number $\displaystyle \alpha=\frac{187+2\sqrt{13}}{113}$. The Beatty
sequence generated by $\alpha$ is $\{1, 3, 5, 6, 8, 10, 12, 13, 15, 17, 18, 20, 22, 24, 25,...\}$. If we run the MES algorithm
with frequency $k=3$ and defining sequence given by $\delta=\sqrt{13}/2$, we obtain Table \ref{tab:t5}.
\begin{center}
 \begin{minipage}{1\linewidth}
 \centering
  \begin{tabular}{l l l l l l}
$n^{th}$  &   A   &   B   &   C   &   R  & D\\
 1)  &   1   &   2   &   0    &   0   & 1    \\
 2)  &   3   &   4   &   0    &   0   & 3    \\
 3)  &   5   &   7   &   1    &   0   & 5     \\
 4)  &   6   &   9  &    1   &   1    & 7      \\
 5)  &   8   &   11  &   1     &   1    & 9  \\
 6)  &   10   &  14  &   2    &   1    &10     \\
 7)  &   12  &   16  &   2    &   1   &12    \\
 8)  &   13  &   19  &   3    &   2   &14     \\
 9)  &   15  &   21  &   3    &   2   &16   \\
 10)  &   17  &  23  &   3    &   2   &18 \\
 11)  &   18  &  26  &   4    &   3    & 19    \\
 12)  &   20  &  28  &   4    &   3   &21   
\end{tabular}\\
\vspace{.3cm}

 \label{tab:t5}
This Table shows the MES algorithm when the frequence $k$ is given by $k=3$ and the defining sequence $D$ is given by $\delta=\sqrt{13}/2$. Note that the output
sequences $A$ is given by $\displaystyle \alpha=\frac{2-\delta + 2 l \delta}{1+k\delta}=\frac{2-\delta + 6 \delta}{1+6\delta}=\frac{187+2\sqrt{13}}{113}$
\end{minipage}\hfil
\vspace{.3cm}
\end{center}

\begin{example}
Example \ref{goldenMESexample} can be rephrased recursively as follows:
 \textit{The skipping sequence $C$ of the MES algorithm that generates the sequence $A$ in Example \ref{goldenMESexample} is
 given by the condition ``if $n \in A$, then it will repeat twice in $C$. Otherwise, $n$ will appears only once in $C$''}. 
Note that if $\delta=\frac{1+\sqrt5}{2}$, and $k=2$ then we get sequence in Example \ref{goldenMESexample}, i.e., we get the sequence with the golden ratio, and this slope coincides with the defining sequence. 
\end{example}

\section{Concluding Remarks}

What does the MES algorithm offer that the MEX does not? Notice that both algorithms allow us to generate any pair of Beatty (and indeed any pair of complementary) sequences. The MEX uses function $b_n-a_n=h_n$ to generate these sequences, and thus any interesting application requires that we are able to easily describe the function $h_n$ independent of $a_n$ and $b_n$. To obtain the Beatty sequences generated $\alpha$ and $\beta$, a few cases are easy to describe: for instance \cite{wythoff1907modification} shows that $h_n=n$ gives the case $\alpha=\frac{1+\sqrt{5}}{2}$, and more generally, \cite{fraenkel_1982} and \cite{holladay1968some} show that $h_n=tn$ gives the family $\frac{2-t+\sqrt{t^2+1}}{2}$. The MES is more advantageous because it gives the general formula $h_n=c_n+r_n+1$ with explicit formulas for $c_n$ and $r_n$ given in the main theorem. It also gives the added benefit of the combinatorial interpretation for $c_n$ and $r_n$ linking it back to the complementary sequences $A$ and $B$. We conclude with an open-ended question: Given the many applications of the mex function or the MEX algorithm as highlighted in section 2, Does the MES algorithm reveal interesting connections that illuminate some aspects of these applications?\\

\noindent {\bf Acknowledgements.}
This research has been partially supported by the Ministerio de Educación Superior, Ciencia y Tecnología (MESCyT) of the Dominican Republic, grant No. 2018-2019-1D2-261

\bibliographystyle{model1a-num-names}

\end{document}